\documentclass[11pt,leqno]{amsart}
\usepackage{amssymb,amsmath,amsthm,amsfonts}
\usepackage{latexsym}
\usepackage{color}
\allowdisplaybreaks

\newcommand{\Hmm}[1]{\leavevmode{\marginpar{\tiny%
$\hbox to 0mm{\hspace*{-0.5mm}$\leftarrow$\hss}%
\vcenter{\vrule depth 0.1mm height 0.1mm width \the\marginparwidth}%
\hbox to 0mm{\hss$\rightarrow$\hspace*{-0.5mm}}$\\\relax\raggedright #1}}}
\newcommand{\nc}{\newcommand}
\nc{\les}{\lesssim}
\nc{\nit}{\noindent}
\nc{\nn}{\nonumber}
\nc{\D}{\partial}
\nc{\diff}[2]{\frac{d #1}{d #2}}
\nc{\diffn}[3]{\frac{d^{#3} #1}{d {#2}^{#3}}}
\nc{\pdiff}[2]{\frac{\partial #1}{\partial #2}}
\nc{\pdiffn}[3]{\frac{\partial^{#3} #1}{\partial{#2}^{#3}}}
\nc{\abs}[1] {\lvert #1 \rvert}
\nc{\cAc}{{\cal A}_c}
\nc{\cE}{{\cal E}}
\nc{\cF}{{\mathcal F}}
\nc{\cP}{{\cal P}}
\nc{\cV}{{\cal V}}
\nc{\cQ}{{\cal Q}}
\nc{\cGin}{{\cal G}_{\rm in}}
\nc{\cGout}{{\cal G}_{\rm out}}
\nc{\cO}{{\cal O}}
\nc{\Lav}{{\cal L}_{\rm av}}
\nc{\cL}{{\cal L}}
\nc{\cB}{{\cal B}}
\nc{\cZ}{{\cal Z}}
\nc{\cR}{{\cal R}}
\nc{\cT}{{\cal T}}
\nc{\cY}{{\cal Y}}
\nc{\cX}{{\cal X}}
\nc{\cXT}{{{\cal X}(T)}}
\nc{\cBT}{{{\cal B}(T)}}
\nc{\vD}{{\vec \mathcal{D}}}
\nc{\efield}{\mathcal{E}}
\nc{\vE}{{\vec \efield}}
\nc{\vB}{{\vec \mathcal{B}}}
\nc{\vH}{{\vec \mathcal{H}}}
\nc{\ty}{{\tilde y}}
\nc{\tu}{{\tilde u}}
\nc{\tV}{{\tilde V}}
\nc{\Pc}{{\bf P_c}}
\nc{\bx}{{\bf x}}
\nc{\bX}{{\bf X}}
\nc{\bXYZ}{{\bf XYZ}}
\nc{\bY}{{\bf Y}}
\nc{\bF}{{\bf F}}
\nc{\bS}{{\bf S}}
\nc{\dV}{{\delta V}}
\nc{\dE}{{\delta E}}
\nc{\TT}{{\Theta}}
\nc{\dPsi}{{\delta\Psi}}
\nc{\order}{{\cal O}}
\nc{\Rout}{R_{\rm out}}
\nc{\eplus}{e_+}
\nc{\eminus}{e_-}
\nc{\epm}{e_\pm}
\nc{\eps}{\varepsilon}
\nc{\vnabla}{{\vec\nabla}}
\nc{\G}{\Gamma}
\nc{\w}{\omega}
\nc{\mh}{h}
\nc{\mg}{g}
\nc{\vphi}{\varphi}
\nc{\tlambda}{\tilde\lambda}
\nc{\be}{\begin{equation}}
\nc{\ee}{\end{equation}}
\nc{\ba}{\begin{eqnarray}}
\nc{\ea}{\end{eqnarray}}

\nc{\g}{\gamma}
\nc{\ol}{\overline}

\newtheorem{theorem}{Theorem}[section]
\newtheorem{lemma}[theorem]{Lemma}
\newtheorem{prop}[theorem]{Proposition}
\newtheorem{corollary}[theorem]{Corollary}
\newtheorem{defin}[theorem]{Definition}

\nc{\pT}{\partial_T}
\nc{\pz}{\partial_z}
\nc{\pt}{\partial_t}
\nc{\la}{\langle}
\nc{\ra}{\rangle}
\nc{\infint}{\int_{-\infty}^{\infty}}
\nc{\halfwidth}{6.5cm}
\nc{\figwidth}{10cm}

\nc{\nlayers}{L} \nc{\nsectors}{M}
\nc{\indicator}{\mathbf{1}}
\nc{\Rhole}{R_{\rm hole}}
\nc{\Rring}{R_{\rm ring}}
\nc{\neff}{n_{\rm eff}}
\nc{\Frem}{F_{\rm rem}}
\nc{\Real}{\mathbb R}
\nc{\Z}{\mathbb Z}
\nc{\DD}{\Delta}
\nc{\cD}{\mathcal D}
\nc{\lnorm}{\left\|}
\nc{\rnorm}{\right\|}
\nc{\rnormp}{\right\|_{\ell^{p,\eps}}}
\nc{\rar}{\rightarrow}
\nc{\sgn}{{\rm sign}}
\allowdisplaybreaks
\date{\today}

\begin{document}
\title[Long time dynamics for forced-damped KdV]{Long time dynamics for forced and weakly damped   KdV on the torus}

\author{M.~B.~Erdo\smash{\u{g}}an and  N.~Tzirakis}
\thanks{The authors were partially supported by NSF grants DMS-0900865 (B.~E.), and DMS-0901222 (N.~T.) }

\address{Department of Mathematics \\
University of Illinois \\
Urbana, IL 61801, U.S.A.}

\email{berdogan@uiuc.edu \\ tzirakis@math.uiuc.edu }

\begin{abstract}
The forced and weakly damped   Korteweg-de Vries (KdV) equation with periodic boundary conditions is considered.
Starting from $L^2$ and mean-zero initial data we prove that the solution decomposes into two parts; a linear one which decays to zero as time goes to infinity and a nonlinear one which always belongs to a smoother space. As a corollary we prove that all solutions are attracted by a ball in $H^s$, $s\in(0,1)$, whose radius depends only on $s$, the $L^2$ norm of the forcing term and the damping parameter.
This gives a  new proof for the existence of a smooth global attractor and provides quantitative information on the size of the attractor set in $H^s$.
\end{abstract}

\maketitle

\section{Introduction}
In this paper we study forced and weakly damped Korteweg de Vries (KdV) equation on the torus:
\begin{align}\label{eq:kdv}
& u_{t}+u_{xxx} +\gamma u + uu_{x}=f, \,\,\,\, t\in {\mathbb R},\,\,  x \in {\mathbb T}=\mathbb R/2\pi\mathbb Z, \\
& u(x,0)=u_{0}(x)\in \dot L^2({\mathbb T}):=\{g\in L^2(\mathbb T): \int_{\mathbb T} g(x) dx=0\},\nn
\end{align}
where throughout the paper $\gamma>0$ and $f\in \dot L^2$. We also assume that $u$ and $f$ are real valued.

Note that forced and damped KdV does not satisfy  momentum conservation. However, since
$$
\partial_t\int_{\mathbb T} u(x,t)dx=-\gamma\int_{\mathbb T} u(x,t)dx,\,\,\,\,\,\,\,\,\, \int_{\mathbb T} u(x,0)dx=0,
$$
the solution $u$ is mean-zero at each time.

The local well-posedness theory for \eqref{eq:kdv}  can be derived by making simple modifications in the well-posedness theory of the KdV equation ($\gamma=f=0$).  Recall that for  a Banach space $X$, and starting with initial data $u_{0} \in H^{s}({\mathbb T})$, one says that the equation \eqref{eq:kdv} is locally well-posed, if there exists $T>0$ and a unique solution $u \in X \cap C_{t}^{0}H_{x}^{s}([0,T]\times \mathbb T)$. One also demands that there is continuity with respect to the initial data in the appropriate topology. If $T$ can be taken to be arbitrarily large then the problem is globally well-posed.

Local well-posedness for the KdV equation on the torus for nonsmooth data was first obtained by Bourgain, \cite{Bou2}. He proved that the KdV equation is locally well-posed in $L^{2}(\mathbb T)$. Later, in \cite{kpv}, Kenig, Ponce, Vega extended the  theory to
$H^s$, $s>-\frac12$. The local theory at $H^{-1/2}$ level was established by Colliander, Keel, Staffilani, Takaoka, Tao, \cite{ckstt}. All these results use the contraction mapping principle and the $X^{s,b}$ spaces  of Bourgain. These methods also apply to the equation \eqref{eq:kdv} as we describe in Section~\ref{sec:wellposed}.

The local $L^2$ solutions of Bourgain are in fact global since for the KdV equation the $L^2$ norm is conserved.
In \cite{ckstt}, it was proved that KdV is globally well-posed in
$H^{s}(\mathbb T)$ for any $s \geq -\frac{1}{2}$.   To extend  the local solutions globally in time they used the ``I-method'', developing a theory of almost conserved quantities starting with the energy. Although the initial data have infinite energy they showed that a smoothed out version of the solution cannot increase much in energy  going from one local-in-time interval to another.   In \cite{kt},   Kappeler and   Topalov extended the latter result using the integrability properties of the equation and proved that   KdV admits global solutions in $H^{s}(\mathbb T)$ for any $s \geq -1$. In \cite{bit}, Babin, Ilyin, ant Titi gave a new proof of the $L^2$ theorem of Bourgain using normal form methods. Similar ideas were developed by Shatah \cite{sha}.

For forced and weakly damped KdV, the conservation of energy does not hold. Nevertheless, the energy remains bounded for positive times and in the long run it is  less then $ 2 \|f\|/\gamma$, where $\|\cdot\|:=\|\cdot\|_{L^2(\mathbb T)}$. Therefore global well-posedness for \eqref{eq:kdv} in $L^2$ follows immediately. For the global well-posedness theory below $L^2$ see \cite{tsugawa} and the references therein.

To obtain the energy bound, note that
\begin{align*}
\partial_t\int u^2dx& =2\int u \partial_t u dx
%\\ &
= -2 \int u u_{xxx}dx-\int u\partial_x u^2 dx -2\gamma \int u^2 dx
+2\int f u dx\\ &= -2\gamma \int u^2 dx
+2\int f u dx.
\end{align*}
Setting $h(t)=e^{2\gamma t} \|u\|^2$, and using Cauchy Schwarz inequality in the integral $\int fu dx$, we obtain
$$
h^\prime(t)\leq 2 e^{\gamma t} \|f\| \sqrt{h(t)}.
$$
This implies that
$$
\partial_t \sqrt{h(t)}\leq e^{\gamma t} \|f\| .
$$
Therefore, for positive times
\be\label{L2bound}
\|u(t)\| \leq e^{-\gamma t} \|u_0\| + \frac{\|f\| }{\gamma}(1-e^{-\gamma t}).
\ee
Thus for $t>T=T(\gamma, \|u_0\|,\|f\| )$,  we have $ \|u(t)\|< 2  \|f\| /\gamma $. Also note  that this implies
that the set $\{g:\|g\|\leq \|f\|/\gamma\}$ is invariant under the flow.

We now state the main result of this paper:
\begin{theorem}\label{theo:main1}
Fix $s \in(0,1)$. Consider the  forced and weakly damped  KdV equation \eqref{eq:kdv}
on $\mathbb{T}\times \mathbb R$ with   $u(x,0)=u_0(x)\in\dot L^2$.
Then $u(t)-e^{-\gamma t}e^{tL}u_0\in C^0_tH^{s}_x$ and
\[
\|u(t)- e^{-\gamma t}e^{tL}u_0\|_{H^{s}} \leq  C(s ,\gamma, \|u_0\|, \|f\| ).
\]
where $L=-\partial_x^3 $.
\end{theorem}

This theorem in the case when $\gamma=f=0$ was obtained in \cite{ET}. Theorem~\ref{theo:main1} and \eqref{L2bound}  imply the following
\begin{corollary}\label{cor:thm1}
Fix $s \in(0,1)$. Consider the  forced and weakly damped  KdV equation \eqref{eq:kdv}
on $\mathbb{T}\times \mathbb R$ with $u(x,0)=u_0(x)\in \dot L^2$. Then there exists $T=T(\gamma, \|u_0\| , \|f\|)$ such that for $t\geq T$,
\[
\|u(t)- e^{-\gamma (t-T)}e^{(t-T)L}u(T)\|_{H^{s}} \leq  C(s ,\gamma, \|f\| ).
\]
\end{corollary}
This implies that all $\dot L^2$ solutions are attracted by a ball in $H^s$ centered at zero of radius depending only on $s ,\gamma, \|f\|$.  An upper bound for this radious can be calculated explicitly by keeping track of the constants in our proof.
Moreover, the description of the dynamics is explicit in the sense that after time $T$ the evolution can be written as a sum of the linear evolution which decays to zero exponentially and a nonlinear evolution contained by the attracting ball. We now try to put this corollary into context in the general theory of global attractors.

The problem of global attractors for nonlinear PDEs has generated great interest among engineers, physicists and mathematicians in the last several decades. The theory is concerned with the description of the nonlinear dynamics for a given problem as $t \to \infty$. In particular assuming that one has a well-posed problem for all times we can define the semigroup operator $U(t):u_{0}\in H\to u(t)\in H$ where $H$ is the phase space. We want to describe the long time asymptotics of the solution by an invariant set $X\subset H$ (a global attractor) to which the orbit converges as $t \to \infty$:
$$U(t)X=X, \ \ t\in \Bbb R_+,\,\,\,\,\,\,\,\,\,\,d(u(t), X)\to 0.$$

For dissipative systems there are many results (see, e.g., \cite{temam}) establishing the existence of  a  compact   set that satisfies the above properties. Dissipativity is characterized by the existence of a bounded absorbing set into which all solutions enter eventually. In the case of the equation \eqref{eq:kdv}, the damping parameter $\gamma>0$ makes the system dissipative, c.f. \eqref{L2bound}. Notice that this is in contrast with conservative Hamiltonian systems where the orbits may fill the whole space or regions of it. In some cases  the global attractor is   a ``thin" set, for example,  it may be a finite dimensional set although the phase space is  infinite dimensional. The candidate for the attractor set  is the omega limit set of an absorbing set, $B$, defined by $$\omega(B)=\bigcap_{s \geq 0}\overline{\bigcup_{t\geq s}U(t)B}$$
where the closure is taken on $H$. To describe the history of the problem more accurately we need some definitions.
\begin{defin}
An attractor is a set $\mathcal{A} \subset H$ that is invariant under the flow and possesses an open neighborhood $\mathcal{U}$ such that, for every $u_{0}\in \mathcal{U}$, $d(U(t)u_{0},\mathcal{A})\to 0$ as $t\to \infty$.
\end{defin}
The distance is understood to be the distance of a point to the set $d(x,Y)=\inf_{y\in Y}d(x,y)$. We say that $\mathcal{A}$ attracts the points of $\mathcal{U}$ and we call the largest open such set $\mathcal{U}$ the basin of attraction.
\begin{defin}
We say that $\mathcal{A} \subset H$ is a global attractor for the semigroup $\{U(t)\}_{t \geq 0}$ if $\mathcal{A}$ is a compact attractor whose basin of attraction is   $H$.
\end{defin}

\vspace{3mm}
\noindent
{\bf Remark.}  We should note that the attracting ball in $H^s$ that our theorem provideas is not a global attractor since we don't know whether it is an invariant set.

\vspace{3mm}

To state a general theorem for the existence of a global attractor we need one more definition:
\begin{defin}
Let $\mathcal{B}_0$ be a bounded set of $H$ and $\mathcal{U}$ an open set containing $\mathcal{B}_0$. We say that $\mathcal{B}_0$ is absorbing in $\mathcal{U}$ if   for any bounded $\mathcal{B} \subset \mathcal{U}$ there exists $T=T(\mathcal{B})$ such that for all $t \geq T$, $U(t)\mathcal{B} \subset \mathcal{B}_0$. We also say that $\mathcal{B}_0$ absorbs the bounded subsets of $\mathcal{U}$.
\end{defin}
It is not hard to see that the existence of a global attractor $\mathcal{A}$ for a semigroup $U(t)$ implies that of an absorbing set. For the converse  we cite the following theorem from \cite{temam} which gives a  general criterion for the existence of a global attractor.

\vspace{5mm}

\noindent
{\bf Theorem A.} {\it
We assume that $H$ is a metric space and that the operator $U(t)$ is a continuous semigroup from $H$ to itself for all $t \geq 0$. We also assume that there exists an open set $\mathcal{U}$ and a bounded set $\mathcal{B}_0$ of $\mathcal{U}$ such that $\mathcal{B}_0$ is absorbing in $\mathcal{U}$. If the semigroup $\{U(t)\}_{t\geq 0}$ is asymptotically compact, i.e. for every bounded sequence $x_k$ in $H$ and every sequence $t_k \to \infty$, $\{U(t_k)x_k\}_{k}$ is relatively compact in $H$, then  $\omega( \mathcal{B}_0)$ is a compact attractor with basin  $\mathcal{U}$ and it is maximal for the inclusion relation.}

\vspace{5mm}

For the forced and weakly damped KdV, Ghidaglia in \cite{jde1988} established the existence of a global attractor in $H^2$ for the weak topology. Moreover the attractor has finite $H^1$ dimension. The result can be then upgraded to a result in the strong topology by an argument of Ball, see \cite{ballDCDS2004} and \cite{ghiaJDE1994}. There are usually two steps in proving such a result. The existence of absorbing sets is derived by establishing energy inequalities coming from the equation. To prove the asymptotic compactness of the semigroup one relies again on energy inequalities and the fact that the semigoup is a continuous mapping for the weak topology of $H^2$ (notice that the continuity of $U(t)$ in $H^2$ does not imply the weak continuity i.e. if $u_{0}^{n} \stackrel{w}{\to} u_{0}$ in $H^2$ then $u^{n}( t) \stackrel{w}{\to} u(t)$ in $H^2$). To this end one uses the fact  that   the mapping $U(t)$ is continuous with respect to the $H^{1}$ norm on bounded subsets of $H^!
 2$. For the details see \cite{jde1988}. Ball's argument uses $L^2$ energy equation to upgrade the asymptotic compactness from the weak to the  strong topology.

In \cite{goubet}, Goubet proved the existence of a global attractor on $\dot L^2$ and  concerning its regularity he proved that the global attractor is a compact subset of $H^3$. This was achieved by splitting the solution into two parts, high and low frequencies. The low frequencies are regular and thus in $H^3$, while the high frequencies decays to zero in $L^2$ as time goes to infinity. The existence of a global attractor below $L^2$ was established by Tsugawa in \cite{tsugawa}. The difficulty there lies in the fact that there is no conservation low for the KdV equation below $L^2$. He bypasses this problem by using the method of almost conserved quantities of Colliander {\em et al}  \cite{ckstt}.
In addition he proves that the global attractor below $L^2$ is same as the one obtained by Goubet. One can lower the Sobolev index further, see \cite{yang}. Also see \cite{CabRos} for numerical results on the properties of the attractor set for the large and small values of the damping parameter.

We will prove in Section~\ref{sec:attract} below that the hypothesis of Theorem A can be checked using only Theorem~\ref{theo:main1} and  Corollary~\ref{cor:thm1}. In particular, one does not need to utilize weak topology and Ball's argument. Theorem~\ref{theo:main1} implies asymptotic compactness in the strong topology directly, see Section~\ref{sec:attract}.  Therefore, we obtain the following:

\begin{theorem}\label{thm:attractor}
Consider the  forced and weakly damped KdV equation \eqref{eq:kdv}
on $\mathbb{T}\times \mathbb R$ with  $u(x,0)=u_0(x)\in \dot L^2$. Then the equation possesses a global attractor in $\dot L^2$. Moreover, for any $s\in (0,1)$,  the global attractor is a compact subset of $H^s$, and it is bounded in $H^s$ by a constant depending  only on $s, \gamma$, and $\|f\|$.
\end{theorem}
Note that the new information provided by this theorem is the dependence of the bound of the attractor set in $H^s$ on  $s, \gamma$, and $\|f\|$. Moreover our proof is different and simpler than the previous known proofs on the existence of the attractor.

\subsection{Outline of the proof of Theorem~\ref{theo:main1}}

We say a few words about the method of the proof. Following the argument in \cite{bit}, \cite{ET} we write the equation on the Fourier side and use differentiation by parts to take advantage of the large denominators that appear due to the dispersion. In this particular form the derivative in the nonlinearity is eliminated. The penalty one pays after such a reduction is an increase in the order of   the nonlinearity (in KdV from quadratic to cubic)  and the appearance of resonant terms. Due to the absence of the zero Fourier modes  the bilinear nonlinearity has no resonant terms.   To estimate the new tri-linear term we now   decompose the nonlinearity into resonant and nonresonant terms.  It should be noted that   in the resonant terms the waves interact with no oscillation and hence they are always ``the enemy''.  But it turns out that the nonsmooth resonant terms of the KdV cancel out and the gain of the derivative is more than enough to compensate for the remaining nonli!
 near terms.  For the nonresonant terms, we apply the restricted norm method of Bourgain to the reduced nonlinearity to prove the theorem.

\subsection{Notation}
To avoid the use of multiple constants, we  write $A \lesssim B$ to denote that there is an absolute  constant $C$ such that $A\leq CB$.  We also write $A\approx B$ to denote both $A\lesssim B$ and $B \lesssim A$. We define $\langle \cdot\rangle =1+|\cdot|$.

We define the Fourier sequence of a $2\pi$-periodic $L^2$ function $u$ as
$$u_k=\frac1{2\pi}\int_0^{2\pi} u(x) e^{-ikx} dx, \,\,\,k\in \mathbb Z.$$
With this normalization we have
$$u(x)=\sum_ke^{ikx}u_k,\,\,\text{ and } (uv)_k=u_k*v_k=\sum_{m+n=k} u_nv_m.$$

Note that for a   $\dot L^2$ function $u$, $\|u\|_{H^{s}}\approx \|\widehat u(k) |k|^{s}\|_{\ell^2}$.
For a sequence $u_k$, with $u_0=0$, we will use $\|u\|_{H^{s}}$ notation to denote $\|u_k |k|^s\|_{\ell^2}$.

\vspace{.4cm}
\noindent
{\bf Acknowledgments.}
We thank J.~Colliander for pointing out the connection between our result in \cite{ET} and the problem of global attractors.

\section{Well-posedness theory of forced and weakly damped KdV}\label{sec:wellposed}
We define the $X^{s,b}$ spaces for $2\pi$-periodic KdV via the norm
$$
\|u\|_{X^{s,b}}=\|\langle k\rangle^s \langle \tau-k^3\rangle^b \widehat u(k,\tau)\|_{L^2(dkd\tau)}.
$$
We also define the restricted norm
$$
\|u\|_{X^{s,b}_\delta}=\inf_{ \tilde u=u \text{ on } [-\delta,\delta]}\|\tilde u\|_{X^{s,b}}.
$$
The local well-posedness theory for periodic KdV was established in the space $X^{s,1/2}$.
Unfortunately, this space fails to control the $L^\infty_t H^s_x$ norm of the solution. To remedy this problem and ensure the continuity of KdV flow, the $Y^s$ and $Z^s$ spaces are defined in \cite{GTV} and \cite{ckstt}, based on the ideas of Bourgain  \cite{Bou2} via the norms
$$
\|u\|_{Y^s}=\|u\|_{X^{s,1/2}}+ \| \langle k\rangle^s\widehat u(k,\tau) \|_{L^2(dk)L^1(d\tau)},
$$
$$
\|u\|_{Z^s}=\|u\|_{X^{s,-1/2}}+\Big\|\frac{\langle k\rangle^s\widehat u(k,\tau)}{\langle \tau-k^3\rangle}\Big\|_{L^2(dk)L^1(d\tau)}.
$$
One defines $Y^s_\delta$, $Z^s_\delta$ accordingly.
Note that if $u\in Y^s$ then $u\in L^\infty_t H^s_x$.

\begin{theorem} \label{thm:I1}
The initial value problem \eqref{eq:kdv} is locally and globally well-posed in $L^2$. In particular, $\exists \delta=\delta( \|u_0\|, \gamma, \|f\|)$ and a unique solution
$u\in C([-\delta,\delta];L^2_x(\mathbb T))\cap Y^0_\delta$
with
$$\|u\|_{X^{0,1/2}_\delta}\leq \|u\|_{Y^0_\delta}\leq C \|u_0\|.$$
\end{theorem}

The proof of this theorem follows the arguments in \cite{Bou2} which we briefly sketch.
The following lemma is a collection of statements originated in \cite{Bou2}:
\begin{lemma}\label{lem:freeys} Let $\eta$ be a smooth function supported on $[-2,2]$  which is identically $1$ on $[-1,1]$. For any $\delta<1$, the following a priori estimates hold
\begin{align*}
&\|\eta(t)e^{Lt} u_0\|_{Y^0_\delta}\lesssim\|u_0\|,\\
&\Big\|\eta(t)\int_0^t e^{L(t-s)} F(s) ds\Big\|_{Y^0_\delta}\lesssim \|F\|_{Z^0_\delta},\\
&\|\partial_x (u^2)\|_{Z^0_\delta}\lesssim \|u\|_{X^{0,1/2}_\delta}\|u\|_{X^{0,1/3}_\delta}.
\end{align*}
For any $-1/2< b^\prime<b<1/2$,
$$\|u\|_{X^{0,b^\prime}_\delta}\lesssim \delta^{b-b^\prime}\|u\|_{X^{0,b}_\delta}.$$
Finally,
$$\|u\|_{Z^{0}_\delta}\lesssim \delta^{1-}\|u\|_{Y^{0}_\delta}.$$
\end{lemma}

The only statement in this lemma which is not explicit in \cite{Bou2} is the last one. By the forth estimate
of the lemma, it suffices to consider the second part of the $Z^0_\delta$ norm, which follows from
\begin{align*}
&\Big\|\frac{\widehat{u\eta(t/\delta)}(k,\tau)}{\langle \tau-k^3\rangle}\Big\|_{L^2(dk)L^1(d\tau)} \lesssim
\Big\| \widehat{u\eta(t/\delta)}(k,\tau)\Big\|_{L^2(dk)L^{\infty-}(d\tau)} \Big\|\frac1{\langle \tau-k^3\rangle}\Big\|_{L^{1+}(d\tau)} \\
&\lesssim \| \delta \widehat{\eta}(\delta \tau)\|_{L^{\infty-}(d\tau)}\Big\| \widehat{u}(k,\tau)\Big\|_{L^2(dk)L^1(d\tau)} \lesssim \delta^{1-}\Big\| \widehat{u}(k,\tau) \Big\|_{L^2(dk)L^1(d\tau)}.
\end{align*}
In the first inequality we used Holder, and in the second one we used Youngs inequality in the $\tau$ variable.

To prove Theorem~\ref{thm:I1}, using Duhamel, we write
$$
\Phi u(t)=\eta(t) e^{Lt}u_0-\eta(t)\int_0^t e^{(t-s)L} F(s) ds,
$$
where $F=\frac12\partial_x u^2+\gamma u-f$. Applying the estimates in the lemma one can easily see that $\Phi$
is a contraction on $Y^0_\delta$ and in addition it belongs to $C([0,\delta];L^2)$. Finally the global well-posedness follows from the energy bound \eqref{L2bound}.

We close this section by stating a Strichartz estimate by Bourgain \cite{Bou2} which will be useful in the proof of Theorem~\ref{theo:main1}
\begin{prop}\cite{Bou2} \label{prop:B}For any $\eps>0$ and $b>1/2$, we have
$$\|\chi_{[-\delta,\delta]}(t)u\|_{L^6_{t,x}(\mathbb R \times \mathbb T)}\leq C_{\eps,b} \|u\|_{X^{\eps,b}_\delta}.$$
\end{prop}

\section{Proof of Theorem~\ref{theo:main1}}\label{sec:evol}
Let $v=\big(\widehat{f}/{(i k)^3}\big)^\vee=\partial_x^{-3} f$, and $w=u-v$. Then $w$ satisfies
\begin{equation}\label{w_eqn}
\left\{
\begin{matrix}
w_{t}+w_{xxx} +\gamma w+\gamma v+ \frac12 \partial_x [(w+v)^2]=0, & x \in {\mathbb T}, & t\in {\mathbb R},\\
w(x,0)=u_0(x)-v(x) \in \dot L^2(\mathbb T),
\end{matrix}
\right.
\end{equation}

Using the notation $u(x,t)=\sum_k u_k(t) e^{ikx}$, and $v(x)=\sum_k v_ke^{ikx}$, we write  \eqref{w_eqn} on the Fourier side:
$$\partial_t w_k=ik^3w_k-\gamma w_k-\gamma v_k-\frac{ik}{2}  \sum_{k_1+k_2=k}(w_{k_1}+v_{k_1})
(w_{k_2}+v_{k_2}),\,\,\,\,\,\,w_k(0)=u_k(0) -v_k.$$
Because of the mean zero assumption on $u$ and $f$,  there are no zero harmonics in this equation.
Using the transformations
\begin{align*}
w_k(t)=z_k(t)e^{-\gamma t+ik^3t},\,\,\,\,\,\,\,
v_k(t)=y_k(t)e^{-\gamma t+ik^3t},
\end{align*}
and the identity
$$(k_1+k_2)^3-k_1^3-k_2^3=3(k_1+k_2)k_1k_2,$$
the equation can be written in the form
\be\label{v_eq}
\partial_t z_k=-\gamma y_k-\frac{ik}{2} e^{-\gamma t}   \sum_{k_1+k_2=k}e^{-i3kk_1k_2t}(z_{k_1}+y_{k_1})
(z_{k_2}+y_{k_2}).
\ee

We start with the following proposition which follows from differention by parts.
\begin{prop}\label{thm:dbp}
The system \eqref{v_eq} can be written in the following form:
\begin{align}\label{v_eq_dbp}
&\partial_t[z-e^{-\gamma t}B(z+y,z+y)]_k =e^{-2\gamma t}\rho_k-\gamma y_k+ \gamma e^{-\gamma t} B(z+y,z+y)_k\\
& -\frac13 e^{-\gamma t} B(z+y,\partial_t y-\gamma y)_k +  e^{-2\gamma t}R(z+y)_k,\nonumber
\end{align}
where we define $B(u,v)_0=\rho_0=R(u)_0=0$, and for $k\neq 0$, we define
\begin{align*}
B(u,v)_k&= \frac16\sum_{k_1+k_2=k}\frac{e^{-3ikk_1k_2t}  u_{k_1}v_{k_2}}{k_1k_2}
\\
\rho_k&=-\frac{i}{6k}  |z_{k}+y_{k}|^2 (z_k+y_k)  \\
R(u)_k&=\frac{i}{6}\sum_{\stackrel{k_1+k_2+k_3=k}{(k_1+k_2)(k_1+k_3)(k_2+k_3)\neq 0}} \frac{e^{-3it(k_1+k_2)(k_2+k_3)(k_3+k_1)}}{k_1} u_{k_1}u_{k_2}u_{k_3}
\end{align*}
\end{prop}

\begin{proof}[Proof of Proposition~\ref{thm:dbp}]
Since $e^{-3ikk_1k_2t}=\partial_{t}( \frac{i}{3kk_1k_2}e^{-3ikk_1k_2t})$, using differentiation by parts we can rewrite \eqref{v_eq} as
\begin{align*}
\partial_{t}z_{k}&=-\gamma y_k+ e^{-\gamma t}  \partial_{t}\Big(   \sum_{k_1+k_2=k}\frac{e^{-3ikk_1k_2t} (z_{k_1}+y_{k_1})(z_{k_2}+y_{k_2})  }{ 6 k_1k_2}\Big)\\
&-
e^{-\gamma t}  \sum_{k_1+k_2=k}\frac{e^{-3ikk_1k_2t}(z_{k_1}+y_{k_1})\partial_t(z_{k_2}+y_{k_2})  }{ 3 k_1k_2}=\\
&=-\gamma y_k+ \partial_{t} \Big (e^{-\gamma t} \sum_{k_1+k_2=k}\frac{e^{-3ikk_1k_2t} (z_{k_1}+y_{k_1})(z_{k_2}+y_{k_2})  }{ 6 k_1k_2}\Big)\\
&
+\gamma e^{-\gamma t} \sum_{k_1+k_2=k}\frac{e^{-3ikk_1k_2t} (z_{k_1}+y_{k_1})(z_{k_2}+y_{k_2})  }{ 6 k_1k_2}\\
&
-e^{-\gamma t}  \sum_{k_1+k_2=k}\frac{e^{-3ikk_1k_2t}(z_{k_1}+y_{k_1})\partial_t(z_{k_2}+y_{k_2})  }{ 3 k_1k_2}.
\end{align*}
Recalling the definition of  $B$, we can rewrite this equation in the form:
\begin{align}\label{eq:partialv}
\partial_t[z-e^{-\gamma t}B(z+y,z+y)]_k &=-\gamma y_k+ \gamma e^{-\gamma t}B(z+y,z+y)_k\\ \nonumber
&-e^{-\gamma t}  \sum_{k_1+k_2=k}\frac{e^{-3ikk_1k_2t}(z_{k_1}+y_{k_1})\partial_t(z_{k_2}+y_{k_2})  }{ 3 k_1k_2}.
\end{align}
Note that since $z_0=y_0=0$, in the sums above $k_1$ and $k_2$ are not zero. Using \eqref{v_eq}, we have
\begin{align*}
&\sum_{k_1+k_2=k}\frac{e^{-3ikk_1k_2t}(z_{k_1}+y_{k_1})\partial_t(z_{k_2}+y_{k_2})  }{ 3 k_1k_2}  = \sum_{k=k_1+k_2}\frac{e^{-3ikk_1k_2t}(z_{k_1}+y_{k_1})(-\gamma y_{k_2}+\partial_ty_{k_2})  }{ 3 k_1k_2}\\
&-\frac{i}{6} e^{-\gamma t}  \sum_{k=k_1+k_2}\frac{e^{-3ikk_1k_2t}(z_{k_1}+y_{k_1})}{k_1}  \sum _{\mu+\nu=k_2\neq 0}e^{-3itk_2\mu\nu} (z_{\mu}+y_{\mu}) (z_{\nu} +y_\nu)\\
&= \frac13 B(z+y,\partial_t y-\gamma y)_k \\
&-\frac{i}{6} e^{-\gamma t} \sum_{\stackrel{k=k_1+\mu+\nu}{\mu+\nu\neq 0}}\frac{e^{-3it[kk_1(\mu+\nu)+\mu\nu(\mu+\nu)]}}{k_1}(z_{k_1}+y_{k_1}) (z_{\mu}+y_{\mu}) (z_{\nu} +y_\nu).
\end{align*}
Using the identity
$$kk_1+\mu\nu=(k_1+\mu+\nu)k_1+\mu\nu=(k_1+\mu)(k_1+\nu)$$
and  by renaming the variables $\mu\to k_2$, $\nu\to k_3$, we have that
\begin{multline}\label{reso1}
\sum_{k_1+k_2=k}\frac{e^{-3ikk_1k_2t}(z_{k_1}+y_{k_1})\partial_t(z_{k_2}+y_{k_2})  }{ 3 k_1k_2}  =
\frac13 B(z+y,\partial_t y-\gamma y)_k \\ -\frac{i}{6} e^{-\gamma t} \sum_{\stackrel{k=k_1+k_2+k_3}{k_2+k_3\neq 0}}\frac{e^{-3it(k_1+k_2)(k_2+k_3)(k_1+k_3)}}{k_1}(z_{k_1}+y_{k_1}) (z_{k_2}+y_{k_2})(z_{k_3}+y_{k_3}).
\end{multline}
Using \eqref{reso1}, we can rewrite \eqref{eq:partialv} as
\begin{align*}
&\partial_t[z-e^{-\gamma t}B(z+y,z+y)]_k =-\gamma y_k+ \gamma e^{-\gamma t}B(z+y,z+y)_k
-\frac13 e^{-\gamma t} B(z+y,\partial_t y-\gamma y)_k \\
&+ \frac{i}{6} e^{-2\gamma t}
\sum_{\stackrel{k=k_1+k_2+k_3}{k_2+k_3\neq 0}}\frac{e^{-3it(k_1+k_2)(k_2+k_3)(k_1+k_3)}}{k_1}(z_{k_1}+y_{k_1}) (z_{k_2}+y_{k_2})(z_{k_3}+y_{k_3}).
\end{align*}
Note that the set on which the phase on the trilinear term vanishes is the disjoint union of the following  sets
\begin{align*}
S_{1}&=\{k_1+k_2=0\}\cap\{k_3+k_1=0\}\cap \{k_2+k_3\neq 0\}\Leftrightarrow \{k_1=-k,\ k_2=k,\ k_3=k\},\\
S_{2}&=\{k_1+k_2=0\} \cap \{ k_3+k_1\ne 0\} \cap \{k_2+k_3\neq 0\}\Leftrightarrow \{k_1=j,\ k_2=-j,\ k_3=k,\ |j| \neq |k| \},\\
S_{3}&=\{k_3+k_1=0\}\cap\{k_1+k_2\ne 0\}   \cap \{k_2+k_3\neq 0\}\Leftrightarrow \{k_1=j,\ k_2=k,\ k_3=-j,\ |j| \neq |k| \}.
\end{align*}
Thus, using the definition of $R(u)$, we have
\begin{align*}
&\partial_t[z-e^{-\gamma t}B(z+y,z+y)]_k =-\gamma y_k+ \gamma e^{-\gamma t}B(z+y,z+y)_k
-\frac13 e^{-\gamma t} B(z+y,\partial_t y-\gamma y)_k \\
&+  e^{-2\gamma t}R(z+y)_k+ \frac{i}{6} e^{-2\gamma t} \sum_{\ell=1}^{3}\sum_{S_{\ell}}\frac{(z_{k_1}+y_{k_1}) (z_{k_2}+y_{k_2})(z_{k_3}+y_{k_3})}{k_1}.
\end{align*}
The proposition follows if we show that the last sum above is equal to $e^{-2\gamma t}\rho_k$. Note that
\begin{align}\label{reso3}
&\sum_{\ell=1}^{3}\sum_{S_{\ell}}\frac{(z_{k_1}+y_{k_1}) (z_{k_2}+y_{k_2})(z_{k_3}+y_{k_3})}{k_1}=-\frac{|z_{k}+y_{k}|^2 (z_k+y_k) }{k}\\\nonumber
&+(z_k+y_k)\sum_{|j|\neq |k|} \frac{|z_j+y_j|^2}{j}+(z_k+y_k)\sum_{|j|\neq |k|} \frac{|z_j+y_j|^2}{j}.
\end{align}
Here we used $z_j=\overline{z_{-j}}$ and $y_j=\overline{y_{-j}}$. Using symmetry, the second line vanishes,
which yields the assertion of the proposition.
\end{proof}

Integrating \eqref{v_eq_dbp} from $0$ to $t$, we obtain
\begin{align*}
&z_k(t)-z_k(0)=e^{-\gamma t}B(z+y,z+y)_k(t)-B(z+y,z+y)_k(0)\\
&+\int_0^t \Big(e^{-2\gamma r}\rho_k(r)-\gamma y_k(r)+\gamma e^{-\gamma r} B(z+y,z+y)_k(r)-\frac13 e^{-\gamma r} B(z+y,\partial_r y-\gamma y)_k(r)\Big) dr\\
&+ \int_0^t    e^{-2\gamma r}R(z+y)_k(r) dr
\end{align*}
Transforming back to the $w$, $v$ variables,  we have
\begin{align}\nonumber
&w_k(t)-e^{-\gamma t+ik^3t}w_k(0)  = \mathcal{B}(w+v,w+v)_k(t)-e^{-\gamma t+ik^3t}\mathcal{B}(w+v,w+v)_k(0)\\
&+\int_0^t e^{-\gamma(t-r)+ik^3(t-r)}\Big( \tilde\rho_k(r)-\gamma v_k(r)+\gamma   \mathcal{B}(w+v,w+v)_k(r)-\frac13  \mathcal{B}(w+v,e^{Lr}\partial_r e^{-Lr}v)_k(r)\Big)  dr\nonumber\\
&+\int_0^t e^{-\gamma(t-r)+ik^3(t-r)}\mathcal{R}(w+v)_k(r)dr,\nonumber
\end{align}
where
\begin{align*}
\mathcal{B}(u,v)_k&=\frac16\sum_{k_1+k_2=k}\frac{  u_{k_1} v_{k_2}}{k_1k_2},\\
\tilde\rho_k&=-\frac{i}{6k}  |w_{k}+v_{k}|^2 (w_k+v_k)\\
\mathcal{R}(u)_k&=\frac{i}{6}\sum_{\stackrel{k_1+k_2+k_3=k}{(k_2+k_3)(k_1+k_2)(k_1+k_3)\neq 0}} \frac{u_{k_1}u_{k_2}u_{k_3}}{k_1}.
\end{align*}
Since $u=w+v$, we have
\begin{align}\label{new_u}
&u_k(t)-e^{-\gamma t+ik^3t}u_k(0)  = v_k-e^{-\gamma t+ik^3t}v_k +\mathcal{B}(u,u)_k(t)- e^{-\gamma t+ik^3t} \mathcal{B}(u,u)_k(0)\\
&+\int_0^t e^{(-\gamma+ik^3)(t-r)}\Big(-\frac{i|u_k|^2u_k}{6k}  -\gamma v_k(r)+\gamma   \mathcal{B}(u,u)_k(r)-\frac13  \mathcal{B}(u,e^{Lr}\partial_r e^{-Lr}v)_k(r)\Big)  dr\nonumber\\
&+\int_0^t e^{(-\gamma+ik^3)(t-r)}\mathcal{R}(u)_k(r)dr,\nonumber
\end{align}

\begin{lemma}\label{apriori} For $s\in(0,1)$, we have
$$\|\mathcal{B}(u,v)\|_{H^{s }}\lesssim \|u\| \|v\|, \text{ and }\,\,\,\,\,\,
\Big\|\frac{|u_k|^2u_k}{6k}\Big\|_{H^{s}}  \leq \|u\|^3.
$$
\end{lemma}
\begin{proof}
By symmetry we can assume in the estimate for   $\mathcal{B}(u,v)$ that $|k_1|\geq |k_2|$. Thus, for $s<1$, we have
\begin{align*}
\|\mathcal{B}(u,v)\|_{H^{s}} &\lesssim \Big\|\sum_{k_1+k_2=k,\,\,\,|k_1|\geq|k_2|}\frac{  |u_{k_1} v_{k_2}|  }{ |k_1k_2|}\Big\|_{H^{s}}
\lesssim  \Big\|\sum_{k_1+k_2=k,\,\,\,|k_1|\geq|k_2|}\frac{  |u_{k_1}|  |v_{k_2}| |k|^{s-1}}{|k_2|}\Big\|\\
&\lesssim \Big\|\sum_{k_1+k_2=k,\,\,\,|k_1|\geq|k_2|}\frac{  |u_{k_1}| |v_{k_2}|  }{|k_2|}\Big\|\lesssim \Big\|\frac{ v_{k}   }{k}\Big\|_{\ell^1} \big\|u_{k}\big\|
\lesssim  \|  v\|  \big\|k^{-1}\big\|_{\ell^2} \|u\|\lesssim \|u\|\|v\|.
\end{align*}
In the last line we used Young's inequality and Cauchy-Schwarz.

Now note that for $s<1$,
$$
\Big\|\frac{|u_k|^2u_k}{6k}\Big\|_{H^{s}}\leq \|u^3\|\leq \|u\|^3. \qedhere
$$
\end{proof}

Using the estimates in Lemma~\ref{apriori} in the equation \eqref{new_u}, we obtain (for $s<1$)
\begin{align}\nonumber
&\|u(t)-e^{-\gamma t}e^{tL}g\|_{H^{s}}\lesssim \|v\|_{H^s}+\|u(t)\|^2+\|u(0)\|^2  \\
&+\int_0^t \Big(\| u(r)\|^3 +\|v\|_{H^s}+\|u(r)\|^2+\|u(r)\|\|\partial_r e^{-Lr} v\|\Big) dr  \nonumber \\
&+\Big\|\int_0^t e^{(-\gamma+ik^3)(t-r)}\mathcal{R}(u)_k(r) dr\Big\|_{H^{s}} \nonumber\\
\label{udelta} &\lesssim \|f\|+\|u(t)\|^2 +\|u(0)\|^2 +\int_0^t \Big(\| u(r)\|^3 +\|f\| +\|u(r)\|^2+\|u(r)\|\|f\|\Big) dr  \\
&+\Big\|\int_0^te^{-\gamma(t-r)} e^{L(t-r)}\mathcal{R}(u)(r) \, dr\Big\|_{H^{s}}, \nonumber
\end{align}
where the implicit constant in the second inequality depends on $\gamma$ and $s$, and
$$\mathcal{R}(u)(r,x)=\sum_{k\neq 0} \mathcal{R}( u)_k(r)  e^{ikx}.$$
We also used the facts that $\partial_re^{-Lr}v=-Le^{-Lr}v=-e^{-Lr}f$ and
$$\|v\|_{H^s} =\|f\|_{H^{s-3}}\leq \|f\|.$$
Since our nonlinearity after differentiation by parts is not $uu_x$ anymore, we will be able to avoid the $Y^{s_1}$ and $Z^{s_1}$ spaces. Instead we will use the embedding $X^{s_1,b}\subset L^\infty_t H^{s_1}_x$ for $b>1/2$ and the following lemma.  Let $\eta$ be a smooth function supported on $[-2,2]$ and $\eta(t)=1$ for $|t|\leq 1$.
\begin{lemma} \label{lem:duhamel}
Let $-\frac12<b^\prime\leq 0$ and $b=b^\prime+1$. Then for any $\delta<1$
\be\label{lem:xsb}\Big\|\eta(t)\int_0^t e^{-\gamma(t-r)}e^{L(t-r)} F(r) dr\Big\|_{X^{s,b}_\delta}\lesssim \|F\|_{X^{s,b^\prime}_\delta},
\ee
where the implicit constant depends on $\gamma$ and $b$.
\end{lemma}
\begin{proof}
It suffices to prove the statement with $X^{s,b}$ norms. Note that
$$
\Big\|\eta(t)\int_0^t e^{-\gamma(t-r)}e^{L(t-r)} F(r) dr\Big\|_{X^{s,b}}=\Big\|\eta(t)\int_0^t e^{-\gamma(t-r)}e^{-Lr} F(r) dr\Big\|_{
H^{s}_xH^b_t}.
$$
Therefore it suffices to prove that
\begin{equation}\label{bbprime}
\Big\|\eta(t)\int_0^t e^{-\gamma(t-r)} f(r) dr\Big\|_{H^b}\lesssim \|f \|_{H^{b^\prime}}.
\end{equation}
Writing
\begin{align*}
\int_0^t e^{-\gamma(t-r)} f(r)dr
%=\int\chi_{[0,t]}(r) e^{-\gamma(t-r)}f(r)dr
=\int \big[\chi_{[0,t]}e^{-\gamma(t-r)}\big]^\vee(z)\widehat{f}(z) dz
=\int \frac{e^{izt}-e^{-\gamma t}}{\gamma+iz}\widehat{f}(z)dz,
\end{align*}
we see that the Fourier transform of the function inside the norm in the left hand side of \eqref{bbprime} is
$$
\int\frac{\widehat{\eta}(\tau-z)-\widehat{\eta e^{-\gamma \cdot}}(\tau)}{\gamma+iz}\widehat{f}(z)dz.
$$
For the contribution of this to the left hand side of \eqref{bbprime}, we use the inequalities
$$\langle \tau\rangle^b\lesssim \langle \tau-z\rangle^b\langle z\rangle^b,\,\,\,\,\,\,\,\,\,\,\frac{1}{|\gamma+iz|}\lesssim \frac{1}{\langle z\rangle},$$
and Young's inequality to get
\begin{align*}
&\Big\|\langle \tau\rangle^b \int\frac{\widehat{\eta}(\tau-z)-\widehat{\eta e^{-\gamma \cdot}}(\tau)}{\gamma+iz}\widehat{f}(z)dz\Big\|_{L^2} \leq \Big\| \langle \tau\rangle^b\int  \frac{|\widehat\eta(\tau-z)|+|\widehat{\eta e^{-\gamma \cdot}}(\tau)|}{\langle z\rangle}  \, |\widehat{f}(z)|dz\Big\|_{L^2}\\
&\lesssim \Big\| \int  \Big(\frac{\langle \tau-z\rangle^b|\widehat\eta(\tau-z)|}{\langle z\rangle^{1-b}}+\frac{\langle \tau\rangle^b|\widehat{\eta e^{-\gamma \cdot}}(\tau)|}{\langle z\rangle} \Big) |\widehat{f}(z)|dz\Big\|_{L^2} \\
&\lesssim \big\|\langle \tau\rangle^b \widehat\eta\big\|_{L^1}
\big\|\langle z\rangle^{b-1}\widehat{f}\big\|_{L^2}+
\big\|\langle \tau\rangle^b \widehat{\eta e^{-\gamma \cdot}}\big\|_{L^2}\big\|\langle z\rangle^{-1}\widehat{f}\big\|_{L^1}
\\
&\lesssim  \big\|\langle z\rangle^{b^\prime}\widehat{f}\big\|_{L^2}+
 \big\|\langle z\rangle^{b^\prime}\widehat{f}\big\|_{L^2}\big\|\langle z\rangle^{-1-b^\prime}\big\|_{L^2}
 \lesssim \|f\|_{H^{b^\prime}}.
\end{align*}
The forth inequality holds since $\eta(t)e^{-\gamma t}$ is a Schwarz function.
The last inequality follows from the fact that  $-1-b^\prime<-1/2$.
\end{proof}

For $|t|<\delta$, where $\delta$  as in Theorem~\ref{thm:I1}, and $b>1/2$,
\begin{align} \label{udelta2}
&\Big\|\int_0^t e^{-\gamma(t-r)}e^{L(t-r)}\mathcal{R}( u)(r)\, dr\Big\|_{H^{s }}
%\\ \nonumber&
\leq
\Big\|\eta(t)\int_0^t e^{-\gamma(t-r)}e^{L(t-r)}\mathcal{R}( u)(r) \,dr\Big\|_{L^\infty_{t\in [-\delta,\delta]} H^{s}_x} \\\nonumber
& \lesssim \Big\|\eta(t)\int_0^t e^{-\gamma(t-r)}e^{L(t-r)}\mathcal{R}( u)(r) \,dr\Big\|_{X^{s ,b}_\delta}
%\\ \nonumber &
\lesssim   \|\mathcal{R}( u)\|_{X^{s ,b-1}_\delta}.
\end{align}
\begin{prop}\label{prop:nonlin} For $s\in(0,1)$,  and $\eps>0$ sufficiently small, we have
$$\|\mathcal{R}(u ) \|_{X^{s,-\frac12+\eps}_\delta}\leq C \|u\|_{X^{0,1/2}_\delta}^3.$$
\end{prop}
We will prove this proposition later on. Using \eqref{udelta2} and the proposition above  in
\eqref{udelta}, we see that for $|t|<\delta $, we have
\begin{align*}
\|u(t)-e^{-\gamma t}e^{tL}u_0\|_{H^{s}}&\lesssim
\|f\|+\|u(t)\|^2 +\|u_0\|^2 +\int_0^t \Big(\| u(r)\|^3 +\|f\| +\|u(r)\|^2+\|u(r)\|\|f\|\Big) dr  \\
&+\|u\|_{X^{0,1/2}_\delta}^3.
\end{align*}
In the rest of the proof the implicit constants depend on $\|u_0\|, \|f\|, \gamma$, and $s$.
Fix $t$ large.  For $r\leq t$, we have the bound
$$\|u(r)\| \leq \|u_0\|+\|f\|/\gamma.$$
Thus, by the inequality above and the local theory, with $\delta \approx \langle \|u_0\|+\|f\|/\gamma \rangle^{-\alpha}$ (for some $\alpha>0$), we have
$$
\|u(j\delta)-e^{\delta (L-\gamma)}u((j-1)\delta)\|_{H^{s}}\lesssim 1,
$$
for any $j\leq t/\delta$. Here we used the local theory bound
$$
\|u\|_{X^{0,1/2}_{[(j-1)\delta,\,j\delta]}} \lesssim \|u((j-1)\delta)\|\leq \|u_0\|+\|f\|/\gamma.
$$
Using this we obtain (with $J=t/\delta$)
\begin{align*}
&\|u(J\delta)-e^{J \delta (L-\gamma)}u(0)\|_{H^{s}} \leq \sum_{j=1}^J\|e^{(J-j) \delta (L-\gamma)}u(j\delta)-e^{(J-j+1) \delta (L-\gamma)}u((j-1)\delta)\|_{H^{s}}\\
&= \sum_{j=1}^Je^{-(J-j) \delta \gamma}\| u(j\delta)-e^{ \delta (L-\gamma)}u((j-1)\delta)\|_{H^{s}}\lesssim  \sum_{j=1}^Je^{-(J-j) \delta \gamma} \lesssim \frac{1}{1-e^{-\delta\gamma}}.
\end{align*}
This completes the proof of the global bound stated in Theorem~\ref{theo:main1}. Finally the continuity of $u(t)-e^{t(L-\gamma)}g$ in $H^{s}$ follows as in \cite{ET}, we omit the details.

\section{Proof of Proposition~\ref{prop:nonlin}}
Recall that
$$\mathcal{R}(u )(r,x)=\sum_{k\neq 0} \mathcal{R}(u )_k(r)  e^{ikx}.$$
We need to prove  that
$$
\|\mathcal{R}(u )\|_{X^{s,-1/2+\eps}_\delta}\lesssim \|u\|_{X^{0,1/2}_\delta}^3.
$$
As usual this follows by considering the  $X^{s,b}$ norms instead of the restricted versions.
By duality it suffices to prove that
\begin{align}\label{dual}
\Big|\sum_k \int_{\mathbb R}\widehat{\mathcal R }(k,\tau)\widehat g(-k,-\tau)d\tau \Big|&=\Big|\int_{\mathbb R\times \mathbb T} \mathcal R(u ) g \Big| \lesssim \|u\|_{X^{0,1/2}}^3 \|g\|_{X^{-s,1/2-\eps}}.
\end{align}
We note that
$$
\widehat{\mathcal{R}}(k,\tau)=\frac{i}{3}\int_{\tau_1+\tau_2+\tau_3=\tau}\sum_{\stackrel{k_1+k_2+k_3=k}{(k_2+k_3)(k_1+k_2)(k_1+k_3)\neq 0}} \frac{\hat u(k_1,\tau_1)\hat u(k_2,\tau_2)\hat u(k_3,\tau_3)}{k_1}.
$$
Let
\begin{align*}
f_1(k,\tau)&=f_2(k,\tau)=f_3(k,\tau)=|\widehat u(k,\tau)|  \langle \tau-k^3\rangle^{1/2},\\
f_4(k,\tau)&=|\widehat g(k,\tau)| |k|^{-s} \langle \tau-k^3\rangle^{1/2-\eps}.
\end{align*}
Note that \eqref{dual} follows from
\be\label{dual1}
\sum_{\stackrel{k_1+k_2+k_3+k_4=0}{(k_2+k_3)(k_1+k_2)(k_1+k_3)\neq 0}} \int_{\tau_1+\tau_2+\tau_3+\tau_4=0}\frac{ |k_4|^{s }\prod_{i=1}^4 f_i(k_i,\tau_i)}{|k_1|  \prod_{i=1}^4\langle\tau_i-k_i^3\rangle^{1/2-\eps}}\lesssim \prod_{i=1}^4\|f_i\|_2.
\ee
By Proposition~\ref{prop:B}, we have (for any $\eps>0$)
\be\label{L6}
\Big\|\Big(\frac{f_i |k|^{-\eps}}{\langle \tau-k^3\rangle^{1/2+\eps}}\Big)^\vee \Big\|_{L^6(\mathbb R \times \mathbb T)}\lesssim \|f_i\|_2.
\ee
Using $\tau_1+\tau_2+\tau_3+\tau_4=0$ and $k_1+k_2+k_3+k_4=0$, we have
$$
\sum_{i=1}^4 \tau_i-k_i^3=-k_1^3-k_2^3-k_3^3-k_4^3=3(k_1+k_2)(k_1+k_3)(k_2+k_3).
$$
Therefore
$$
\max_{i=1,2,3,4} \langle \tau_i-k_i^3 \rangle \gtrsim |k_1+k_2| |k_1+k_3| |k_2+k_3|.
$$
Since the inequality \eqref{dual1} is symmetric in $f_i$'s, it does not matter which of these terms is the maximum. Therefore without loss of generality we assume that
$$
\langle \tau_1-k_1^3\rangle =\max_{i=1,2,3,4} \langle \tau_i-k_i^3 \rangle  \gtrsim |k_1+k_2| |k_1+k_3| |k_2+k_3|.
$$
This implies that
\be\label{maxtau}
\prod_{i=1}^4\langle\tau_i-k_i^3\rangle^{1/2-\eps}\gtrsim (|k_1+k_2| |k_1+k_3| |k_2+k_3|)^{1/2-7\eps}\prod_{i=2}^4\langle\tau_i-k_i^3\rangle^{1/2+\eps}.
\ee
Also note that (since all factors are nonzero and $k_1+k_2+k_3+k_4=0$)
\be\label{kis}
|k_1+k_2| |k_1+k_3| |k_2+k_3|\gtrsim |k_i|,\,\,\,\,i=1,2,3,4.
\ee
Now we will prove that for $s\in(0,1)$, and for $\eps$ sufficiently small,
\be\label{multiplier}
\frac{ |k_4|^{s } }{|k_1| (|k_1+k_2| |k_1+k_3| |k_2+k_3|)^{1/2-7\eps}} \lesssim |k_1k_2k_3k_4|^{-\eps}.
\ee
By \eqref{kis}, this follows from
$$\frac{ |k_4|^{s } }{|k_1| (|k_1+k_2| |k_1+k_3| |k_2+k_3|)^{1/2-11\eps}}\lesssim 1.
$$
To prove this let $M=\max(|k_1|,|k_2|,|k_3|)$. Using $|k_1||k_1+k_2|\gtrsim|k_2|$ and $|k_1||k_1+k_3||k_3+k_2|\gtrsim|k_2|$, and by symmetry of $k_2, k_3$, we have
$$|k_1| (|k_1+k_2| |k_1+k_3| |k_2+k_3|)^{1/2-11\eps}\gtrsim M^{1-22\eps}\gtrsim |k_4|^{1-22\eps}.$$

This finishes the proof of \eqref{multiplier}. Using \eqref{multiplier} and \eqref{maxtau} in \eqref{dual1} (and eliminating $|k_1|^{-\eps}$), we obtain
\begin{align*}
\eqref{dual1}&\lesssim   \sum_{k_1+k_2+k_3+k_4=0 } \int_{\tau_1+\tau_2+\tau_3+\tau_4=0}\frac{|k_2k_3k_4|^{-\eps} \prod_{i=1}^4 f_i(k_i,\tau_i)}{  \prod_{i=2}^4\langle\tau_i-k_i^3\rangle^{1/2+\eps}}.
\end{align*}
By Plancherel, and the convolution structure we can rewrite this as
\begin{align*}
\int_{\mathbb R \times \mathbb T}\widehat{f_1}(x,t) \prod_{i=2}^4\Big(\frac{f_i |k|^{-\eps}}{\langle \tau-k^3\rangle^{1/2+\eps}}\Big)^\vee(x,t) &\leq \|f_1\|_{L^2(\mathbb R \times \mathbb T)} \prod_{i=2}^4\Big\|\Big(\frac{f_i |k|^{-\eps}}{\langle \tau-k^3\rangle^{1/2+\eps}}\Big)^\vee \Big\|_{L^6(\mathbb R \times \mathbb T)}  \\ &
\lesssim \prod_{i=1}^4\|f_i\|_2.
\end{align*}
In the last inequality we used \eqref{L6}.

\section{Proof of Theorem~\ref{thm:attractor}}\label{sec:attract}
First of all note that the existence of an absorbing set, $\mathcal B_0$, is immediate from \eqref{L2bound}. Second, we need to verify the assymptotic compactness of the propagator $U(t)$. To see this note that by Theorem~\ref{theo:main1},
$$U(t)u_0=e^{-\gamma t+Lt}u_0+ N(t)u_0$$
where $N(t)u_0$ is in a ball in $H^s$ with radius depending on $s, \gamma, \|u_0\|, \|f\|$. By Rellich's theorem, $\{N(t)u_0:t>0\}$ is precompact in $L^2$. Since the first summand is continuous and goes to zero as $t \to \infty$ in $L^2$, we conclude that  $\{U(t)u_0:t>0\}$ is precompact in $L^2$, which is stronger then asymptotic compactness. This and theorem A imply the existence of a global attractor $\mathcal A\subset L^2$.

We now prove that the attractor set $\mathcal A$ is a compact subset of $H^s$ for any $s\in(0,1)$. By Rellich's theorem, it suffices to prove that for any $s\in(0,1)$, there exists a closed ball $B_s\subset H^s$ of radius $C(s,\gamma,\|f\|)$ such that  $\mathcal A\subset B_s$.
By definition
$$
\mathcal A=\bigcap_{\tau\geq 0}\overline{\bigcup_{t\geq \tau} U(t)\mathcal B_0}=: \bigcap_{\tau\geq 0} U_\tau.
$$
Let $B_s$ be the ball of radius $C(s,\gamma,\|f\|)$ (as in Corollary~\ref{cor:thm1}) centered at zero in $H^s$.
By Corollary~\ref{cor:thm1}, for $\tau>T$, $U_\tau$ is contained in a $\delta_\tau$ neighborhood $N_\tau$ of $B_s$ in $L^2$, where $\delta_\tau\to 0$ as $\tau$ tends to infinity. Since   $B_s$ is a compact subset of
$L^2$, we have
$$
\mathcal A = \bigcap_{\tau\geq 0} U_\tau\subset \bigcap_{\tau\geq 0} N_\tau =B_s.
$$

\end{document}